\documentclass[11pt,oneside,a4paper]{article}

\usepackage[latin1]{inputenc}
\usepackage[english]{babel}
\usepackage{amsmath,amssymb,amsthm,epsfig,color,graphicx}
\usepackage{hyperref}
\usepackage{vmargin}

\newtheorem{theorem}{Theorem}
\newtheorem{proposition}[theorem]{Proposition}
\newtheorem{lemma}[theorem]{Lemma}

\title{Subcritical and supercritical Klein-Gordon-Maxwell equations without Ambrosetti-Rabinowitz condition}

\author{Patr\'{i}cia L. Cunha\thanks{Supported by FAPESP/Brazil.}\thanks{patcunha80@gmail.com} \\
\footnotesize{Departamento de Matem\'atica, ICMC-USP, S\~ao Carlos, SP, Brazil} }

\begin{document}
\maketitle

\begin{abstract}
In this article we present some results on the existence of positive and ground state solutions for the nonlinear Klein-Gordon-Maxwell equations.
We introduce a general nonlinearity with subcritical and supercritical growth which does not require the usual Ambrosetti-Rabinowitz condition.
The proof is based on variational methods and perturbation arguments.
\end{abstract}

\text{\footnotesize{\textit{Keywords}: Klein-Gordon-Maxwell equations; Ambrosetti-Rabinowitz condition.}}

\section{Introduction}

The Klein-Gordon-Maxwell system
\begin{equation}\label{eq0} \left \{ \begin{array}{ll}
 -\Delta u+ [m_0^2-(\omega +\phi)^2]u= f(u) &
 \mbox{in}\quad \mathbb{R}^{3}\\
\hspace{0.26cm} \Delta \phi=(\omega+\phi)u^{2} & \mbox{in}\quad
\mathbb{R}^{3}
\end{array}\right.
\end{equation}
arises in a very interesting physical context: as a model describing the nonlinear Klein-Gordon field interacting with the
electromagnetic field. More specifically, it represents a standing wave $\psi=u(x) e^{i\omega t}$ in equilibrium with a purely electrostatic
field $\textbf{E}=-\nabla\phi(x)$.

Benci and Fortunato \cite{Benci-Fortunato-2002} were pioneering on working with this system. They considered $|\omega|<|m_0|$ and $f(u)=|u|^{q-2}u$,
$4<q<2^*=6$, and proved that system (\ref{eq0}) has infinitely many radially symmetric solutions. In \cite{D'Aprile-Mugnai}, D'Aprile and Mugnai
extended the interval of definition of the power in the nonlinearity for the case $2<q\leq 4$. A nonexistence result has been established by the same
authors in \cite{D'Aprile-Mugnai-nonexistence}. They proved that any weak solution of (\ref{eq0}) vanishes identically under the conditions
\begin{itemize}
\item[(i)] $f:\mathbb{R}\rightarrow\mathbb{R}$ is a continuous function;
\item[(ii)] For every $s\in\mathbb{R}$, either $sf(s)+2(m_0^2-\omega^2)s^2\geq 6F(s)$ or $2F(s)\geq sf(s)$.
\end{itemize}
In particular, $f(u)=|u|^{2^*-2}u$ and $m_0>\omega$ satisfies (i) and (ii).

Motivated by the above results, Cassani \cite{Cassani} proved the existence of nontrivial radially symmetric solutions in $\mathbb{R}^{3}$ for the
critical case $f(u)=|u|^{q-2}u+|u|^{2^*-2}u$, $4\leq q< 2^*=6$. Afterwards, Carrião, Cunha $\&$ Miyagaki \cite{Carriao-Cunha-Miyagaki} completed the result
of Cassani \cite{Cassani} considering the case $2<q<4$.

In \cite{Azzollini-Pomponio-KGM}, the authors Azzollini and Pomponio showed that system (\ref{eq0}) admits a ground state solution with
subcritical exponents $q$ in the range $2<q<6$. Taking into account \cite{Azzollini-Pomponio-KGM}, Carrião, Cunha $\&$
Miyagaki \cite{Carriao-Cunha-Miyagaki-2} proved the existence of positive ground state solutions for the critical Klein-Gordon-Maxwell system with
a potential $V$, namely,
\begin{eqnarray*} \left \{ \begin{array}{ll}
-\Delta u+ V(x)u-(2\omega +\phi)\phi u=\mu|u|^{q-2}u +|u|^{2^*-2}u, & \mbox{in}\quad \mathbb{R}^{3}\\
\hspace{0.26cm} \Delta \phi=(\omega+\phi)u^{2}, & \mbox{in}\quad \mathbb{R}^{3}
\end{array}\right.
\end{eqnarray*}

We observe that this class of Klein-Gordon-Maxwell system with such potential $V(x)$ is closely related to a number of several other works. In fact,
the potential $V(x)$ considered also satisfies the constant case $m_0^2-\omega^2$ which has been extensively treated, see
e.g. \cite{Azzollini-Pomponio-KGM,Azzollini-Pisani-Pomponio,Benci-Fortunato-2002,Cassani,D'Aprile-Mugnai,D'Aprile-Mugnai-nonexistence}.

This article is concerned with the Klein-Gordon-Maxwell system written as
\begin{equation}\label{kgm} \left \{ \begin{array}{ll}
 -\Delta u+ V(x)u-(2\omega +\phi)\phi u=f(u), & \mbox{in}\quad \mathbb{R}^{3}\\
\hspace{0.26cm} \Delta \phi=(\omega+\phi)u^{2}, & \mbox{in}\quad \mathbb{R}^{3}
\end{array}\right. \tag{$\mathcal{KGM}$}
\end{equation}
where $V\in C(\mathbb{R}^3,\mathbb{R})$ satisfies
\begin{itemize}
\item[($V_1$)] there exists $\alpha>0$ such that $V(x)\geq\alpha$, for all $x\in\mathbb{R}^N$;
\item[($V_2$)] $V(x)=V(x+y)$, for all $ x\in\mathbb{R}^3,\,y\in\mathbb{Z}^3$.
\end{itemize}

In the first part of this paper we assume the following conditions on the nonlinearity $f\in C^1(\mathbb{R}^+,\mathbb{R})$
\begin{itemize}
\item[($f_1$)] $f(0)=0$;
\item[($f_2$)] $\displaystyle\lim_{s\rightarrow 0^+}\frac{f(s)}{s}=0$;
\item[($f_3$)]  There exist $C>0$ and $p\in(4,6)$ such that $|f(s)|\leq C(s+|s|^{p-1})$, for all $ s \in\mathbb{R}^+$.
\end{itemize}

To the best of our knowledge the only result considering a general nonlinearity in the Klein-Gordon-Maxwell system is due to Azzollini, Pisani
$\&$ Pomponio \cite{Azzollini-Pisani-Pomponio}. They proved the existence of weak solutions for the limit case $\omega=m_0$, and also for
$m_0>\omega>0$, assuming among others hypothesis that $f$ satisfies the Ambrosetti-Rabinowitz condition:
\begin{itemize}
\item [(AR)] There exists $\theta>4$ and $s_0>0$ such that  $0<\theta F(s)\leq sf(s)$, for all $s\in\mathbb{R}$, \\ where $F(s)=\int_{0}^{s}f(t)\,dt$.
\end{itemize}

The condition (AR) is widely assumed in the studies of elliptic equation by variational methods. As it is known, the condition (AR) is employed
not only to show that the Euler-lagrange Euler-Lagrange functional associated has a mountain pass geometry, but also to guarantee that the Palais-Smale,
or Cerami, sequences are bounded. Although (AR) is a quite natural condition, it is somewhat restrictive because it eliminates many nonlinearities
such as $f(s)=s^3(4\ln(1+|s|)+s/(1+|s|))$.

Observe that (AR) implies the weaker condition there exist $\theta>4$ and constant $C_1, C_2>0$ such that $F(s)\geq C_1|s|^\theta-C_2$, for every $|s|$,
sufficiently large, which, in its turn, implies another much weaker one
\begin{itemize}
\item[($f_4$)] $\displaystyle\lim_{s\rightarrow +\infty}\frac{F(s)}{s^4}=+\infty$.
\end{itemize}

In order to prove that the Euler-Lagrange functional associated with the  $(\mathcal{KGM})$ system possesses a bounded Cerami sequence, we also
assume that the nonlinearity $f$ satisfies
\begin{itemize}
\item[($f_5$)] $\frac{f(s)}{s^3}$ is increasing in $|s|>0$.
\end{itemize}

In \cite{Alves-Souto-Soares} Alves, Soares $\&$ Souto considered the Schr\"{o}dinger-Poisson system
\begin{eqnarray} \label{SP}\left \{ \begin{array}{ll}
-\Delta u+ V(x)u+\psi u=f(u), & \mbox{in}\quad \mathbb{R}^{3}\\
-\Delta \psi=u^{2}, & \mbox{in}\quad \mathbb{R}^{3}
\end{array}\right.
\end{eqnarray}
under conditions $(f_1)-(f_5)$ and assuming periodicity in $V$, they proved that the Schr\"{o}dinger-Poisson system possesses a positive ground state solution.
Besides they substitute condition $(f_5)$ by a weaker one
\begin{itemize}
\item [(${f'_5}$)] $H(s)=sf(s)-4F(s)\geq 0$, for all $ s\in\mathbb{R}$,
\end{itemize}
and proved that the Shr\"{o}dinger-Poisson system has a positive solution.

In the same spirit of \cite{Alves-Souto-Soares}, we prove that their results apply to the $(\mathcal{KGM})$ system with the periodicity condition on
$V$ as well. In our case, however, many technical difficulties arise due to the presence of a non-local term $\phi$, which is not homogeneous as it is in
the Schr\"{o}dinger-Poisson system. Hence, a more careful analysis of the interaction between the couple $(u,\phi)$ is required.

Our first result is
\begin{theorem}\label{teoremaA} Suppose $V$ satisfies $(V_1)-(V_2)$ and $f$ satisfies $(f_1)-(f_5)$. Then the $(\mathcal{KGM})$ system has a positive
ground state solution.
\end{theorem}

As a consequence, we have
\begin{theorem}\label{teoremaB} Under the hypotheses of Theorem \ref{teoremaA}, with $(f_5)$ replaced by $(f'_5)$,
the $(\mathcal{KGM})$ system has a positive solution $u$ such that $\|u\|^2\leq 4c$, where $c$ is the minimax level associated with the $(\mathcal{KGM})$
equations.
\end{theorem}

We observe that condition $(f_5)$ implies that $H(s)=sf(s)-4F(s)$ is a non-negative function and that $H(s)$ is increasing in $s$, as we can see
in \cite[Lemma 2.2]{Liu}. Hence, condition $(f'_5)$ is weaker than $(f_5)$ and also weaker than $(AR)$, whereas $(f_5)$ and (AR) are just different.

Under the setting of Theorem \ref{teoremaB} and motivated by the paper of Jeanjean and Tanaka \cite{Jeanjean-Tanaka-asypmtoticallylinear},
we can also show the existence of a ground state solution 

\begin{theorem}\label{teoremaB2} Suppose $V$ satisfies $(V_1)-(V_2)$ and $f$ satisfies $(f_1)-(f_4)$ and $(f'_5)$. Then the (\ref{kgm}) system has a 
ground state solution.
\end{theorem}

In the second part of the paper we consider the function $f$ written as:
\begin{eqnarray*}f(s)=f_0(s)+\lambda g(s),
\end{eqnarray*}
where $\lambda$ is a positive real parameter, $f_0$ and $g$ are locally H\"{o}lder continuous functions satisfying:
\begin{itemize}
\item [$(F_1)$] $f_0(0)=g(0)=0$ and $g(s)\geq 0$, for all $s\in\mathbb{R}$;
\item [$(F_2)$] $\displaystyle\lim_{s\rightarrow 0^+}\frac{f_0(s)}{s}=\lim_{s\rightarrow 0^+}\frac{g(s)}{s}=0$;
\item [$(F_3)$] There exists $q\in (4,6)$ such that $|f_0(s)|\leq|s|^{q-1}$, for all $s\in\mathbb{R}$;
\item [$(F_4)$] $\displaystyle\lim_{s\rightarrow +\infty}\frac{F_0(s)}{s^4}=+\infty$, where $F_0(s)=\int_{0}^{s}f_0(t)\,dt$;
\item [$(F_5)$] $sf_0(s)-4F_0(s)\geq 0$  and  $sg(s)-4G(s)\geq 0$,\, for all $s\in\mathbb{R}$, where $G(s)=\int_{0}^{s}g(t)\,dt$;
\item [$(F_6)$] There exists a sequence of positive real numbers $(M_n)$ converging to $+\infty$ such that
\begin{eqnarray*} \frac{g(s)}{s^{q-1}}\leq\frac{g(M_n)}{M_n^{q-1}},\quad \mbox{for all}\, \, s\in [0,M_n],\, n\in\mathbb{N}.
\end{eqnarray*}
\end{itemize}
then, the following result holds
\begin{theorem}\label{teoremaC} Suppose that $V$ satisfies $(V_1)-(V_2)$ and $f$ satisfies $(F_1)-(F_6)$. Then there exists $\lambda_0>0$ such that
the (\ref{kgm}) system has a positive solution for all $\lambda\leq\lambda_0$.
\end{theorem}

It is noteworthy that to prove Theorem \ref{teoremaC} we do not require any growth assumption on $g$, consequently, on $f$. We observe that
condition $(F_6)$ holds if
\begin{eqnarray*}\lim_{|s|\rightarrow +\infty}\frac{g(s)}{s^{q-1}}=+\infty.
\end{eqnarray*}
In particular, $f$ may be $f(s)=s^{q-1}+s^{p-1}$, for all $q<6<p$, or $f(s)$ may behave like $e^{s}$ at infinity.

For this situation involving supercritical growth, we cite the work of Alves, Soares $\&$ Souto \cite{Alves-Souto-Soares-supercritical} where
they studied the Schr\"{o}dinger-Poisson system (\ref{SP}).

In the proof of Theorem \ref{teoremaC} we use similar arguments to those used in \cite{Alves-Souto-Soares-supercritical}. The strategy consists in
combining perturbation arguments, estimates of solutions for the subcritical Klein-Gordon-Maxwell system in terms of the $L^{\infty}$ norm and the
mountain pass theorem.

\section{Preliminaries}

We denote by $E$ the Sobolev space $H^{1}(\mathbb{R}^3)$ endowed with the norm
\begin{eqnarray*} \|u\|^{2}=\int_{\mathbb{R}^{3}}(|\nabla u|^2+V(x)u^2)\,dx,
\end{eqnarray*}
\noindent which is equivalent to the usual Sobolev norm on $H^1(\mathbb{R}^3)$. $\mathcal{D}^{1,2}\equiv\mathcal{D}^{1,2}(\mathbb{R}^{3})$
represents the completion of $\mathcal{C}_{0}^{\infty}(\mathbb{R}^{3})$ with respect to the norm
\begin{eqnarray*} \|u\|_{\mathcal{D}^{1,2}}^{2}=\int_{\mathbb{R}^{3}}|\nabla u|^2\,dx.
\end{eqnarray*}

The solutions $(u,\phi)\in E\times\mathcal{D}^{1,2}$ of the (\ref{kgm}) system are critical points of the functional
$J:E\times \mathcal{D}^{1,2} \rightarrow \mathbb{R}$ defined as
\begin{eqnarray*} && J(u,\phi) =\frac{1}{2}\int_{\mathbb{R}^{3}} (|\nabla u|^2-|\nabla\phi|^2+[V(x)-(2\omega+\phi)\phi]u^2)\,dx
-\int_{\mathbb{R}^{3}}F(u)\, dx,
\end{eqnarray*}
\noindent which by standard arguments is $C^{1}$ on $E\times \mathcal{D}^{1,2}$.

As it has been done by the aforementioned authors, we apply a \textit{reduction method} in order to avoid the difficulty originated by the
strongly indefiniteness of the functional $J$.

\begin{proposition}\label{Propriedade-phi}
For every $u \in E$, there exists a unique $\phi=\phi_u\in \mathcal{D}^{1,2}$ which solves $\Delta \phi=(\omega+\phi)u^{2}$. Furthermore,
in the set $\{x:\,u(x)\neq 0\}$ we have $-\omega\leq\phi_u\leq 0$ for $\omega>0$.
\end{proposition}

\begin{proof} The proof can be found in \cite{Benci-Fortunato-2002,Carriao-Cunha-Miyagaki,D'Aprile-Mugnai-nonexistence}.
\end{proof}

According to Proposition \ref{Propriedade-phi}, we can define
\begin{eqnarray*} \Phi :E\rightarrow \mathcal{D}^{1,2}
\end{eqnarray*}
which maps each $u\in E$ in the unique solution of $\Delta \phi=(\omega+\phi)u^{2}$. From standard arguments it results $\Phi\in C^1(E,\mathcal{D}^{1,2})$.

Hence, we have
\begin{eqnarray}\label{system2-modificado}-\Delta \phi_u + u^{2} \phi_u
= -\omega u^{2}.
\end{eqnarray}

Multiplying both members of (\ref{system2-modificado}) by $\phi_u$ and integrating by parts, we obtain
\begin{eqnarray}\label{system2-modificado2}\int_{\mathbb{R}^{3}}
{| \nabla\phi_u|^{2}}dx = -\int_{\mathbb{R}^{3}}{\omega\phi_u}
u^{2}dx-\int_{\mathbb{R}^{3}}{\phi_u^{2}u^{2}}dx.
\end{eqnarray}

By the definition of $J$ and using (\ref{system2-modificado2}), we obtain a $C^1$ functional $I:E\rightarrow\mathbb{R}$ given by
\begin{eqnarray*}I(u)=\frac{1}{2}\int_{\mathbb{R}^3}(|\nabla u|^2+V(x)u^2)\,dx-
\frac{1}{2}\int_{\mathbb{R}^3}\omega\phi_u u^2\,dx-\int_{\mathbb{R}^3}F(u)\,dx
\end{eqnarray*}
and its Gateaux derivative is
\begin{eqnarray*}I'(u)v=\int_{\mathbb{R}^3}(\nabla u\cdot \nabla v+V(x)uv)\,dx-\int_{\mathbb{R}^3}(2\omega +\phi)\phi uv\,dx
-\int_{\mathbb{R}^3}f(u)v\,dx
\end{eqnarray*}
for every $u,v\in E$. Then, $(u,\phi)\in E\times \mathcal{D}^{1,2}$ is a weak solution of $(\mathcal{KGM})$ if, and only if, $\phi=\phi_u$ and
$u\in E$ is a critical point of $I$, that is, a weak solution of
\begin{eqnarray}\label{equacao1} -\Delta u+V(x)u-(2\omega+\phi_u)\phi_u u=f(u), \quad \mbox{in}\,\,\mathbb{R}^3.
\end{eqnarray}

From $(f_1)-(f_4)$ and Proposition \ref{Propriedade-phi} the functional $I$ satisfies the structural assumptions of the mountain pass theorem,
as we can see in the next lemma.

\begin{lemma}\label{MPG} Suppose that $V$ satisfies $(V_1)$ and $f$ satisfies $(f_1)-(f_4)$. Then, there exist $r>0$ and $e\in E$, $\|e\|>r$ such that
\begin{eqnarray*}b\doteq\inf_{\|u\|=r}I(u)>I(0)=0\geq I(e).
\end{eqnarray*}
\end{lemma}

\begin{proof} From $(f_2)-(f_3)$, given $\epsilon>0$ there exists $C_\epsilon>0$ such that
\begin{eqnarray*}F(s)\leq \epsilon s^2+C_\epsilon s^p, \quad \forall s\in \mathbb{R}.
\end{eqnarray*}

By Sobolev immersions, there exist positive constants $\alpha$ and $\beta$ such that
\begin{eqnarray*}I(u)\geq \Big[ \Big( \frac{1}{2}-\epsilon\alpha \Big)-\beta C_\epsilon\|u\|^{p-2} \Big]\|u\|^2.
\end{eqnarray*}

We can assume, by decreasing $\epsilon$ if necessary, that there exist positive numbers $b$, $r$ such that $b=\inf \{I(u),\|u\|=r\}>I(0)=0$.

From $(f_4)$, for any $\mathrm{v}\in E$ and $M>0$, there exists $C>0$ such that $F(s)\geq Ms^4-Cs^2$, for all $s\in\mathbb{R}$. Hence, from
Proposition \ref{Propriedade-phi},
\begin{eqnarray*}I(t\mathrm{v})&=&\frac{1}{2}\int_{\mathbb{R}^3}[|t\mathrm{v}|^2+(t\mathrm{v})^2]\,dx-
\frac{1}{2}\int_{\mathbb{R}^3}\omega\phi_{t\mathrm{v}}(t\mathrm{v})^2\,dx-\int_{\mathbb{R}^3}F(t\mathrm{v})\,dx\\
&\leq & \frac{t^2}{2}\|\mathrm{v}\|^2-\frac{t^2}{2}\int_{\mathbb{R}^3}\omega\phi_{t\mathrm{v}}\mathrm{v}^2\,dx-
\int_{\mathbb{R}^3}[M(t\mathrm{v})^4-C(t\mathrm{v})^2]\,dx\\
&\leq & \frac{t^2}{2}\|\mathrm{v}\|^2+\frac{t^2}{2}\omega^2\|\mathrm{v}\|_2^2+
Ct^2\|\mathrm{v}\|_{2}^2-Mt^4\\
&\leq & \frac{t^2}{2}\|\mathrm{v}\|^2+t^2 C_1\|\mathrm{v}\|^2-Mt^4 \rightarrow - \infty, \quad \text{as}\,\, t\rightarrow\infty.
\end{eqnarray*}

Then, for $t$ sufficiently large, $e=t\mathrm{v}$ satisfies $\|e\|>r$ and $I(e)<0=I(0)$.
\end{proof}

Now, by using a version of the mountain pass theorem (see \cite{Ekeland}), there is a Cerami sequence $(u_n)\subset E$ such that
\begin{eqnarray*}\label{cerami-sequence}I(u_n)\rightarrow c \qquad \mbox{and}
\qquad
(1+\|u_n\|)I'(u_n)\rightarrow 0,
\end{eqnarray*}
\noindent where
\begin{eqnarray*}\label{valorC}\displaystyle c:=\inf_{\gamma\in\Gamma} \max_{t\in [0,1]}I(\gamma(t))\qquad \mbox{and}
\qquad \Gamma=\{\gamma:[0,1]\rightarrow E : \gamma(0)=0,
\gamma(1)=e \}.
\end{eqnarray*}

\begin{lemma}\label{cerami-bounded}
Suppose that $V$ satisfies $(V_1)$ and $f$ satisfies $(f_1)-(f_4)$ and $(f'_5)$. Then the Cerami sequence $(u_n)\subset E$ for $I$
is bounded.
\end{lemma}

\begin{proof} Using condition $(f'_5)$, we have
\begin{eqnarray*}
4I(u_n)-I'(u_n)(u_n)&=&\|u_n\|^2+\int_{\mathbb{R}^3}\phi_{u_n}^2 u_n^2\,dx + \int_{\mathbb{R}^3}[f(u_n)u_n-4F(u_n)]\\
&\geq & \|u_n\|^2.
\end{eqnarray*}
Once $4I(u_n)-I'(u_n)(u_n)$ is bounded, the last limit implies the boundedness of $(u_n)$.
\end{proof}

Observe that from condition $(f_5)$ and \cite[Lemma 4.1]{Willem},
\begin{eqnarray*}c=\inf_{u\in\mathcal{N}}I(u),
\end{eqnarray*}
where
\begin{eqnarray*}\mathcal{N}=\{u\in E\backslash\{0\}:I'(u)u=0\}.
\end{eqnarray*}

\section{The subcritical case}\label{The subcritical case}

In this section we prove Theorem \ref{teoremaA} and Theorem \ref{teoremaB}.

In view of Lemma \ref{cerami-bounded} we have that $(\phi_{u_n})$ is bounded in $\mathcal{D}^{1,2}$.
Indeed, we have
\begin{eqnarray*}\|\phi_{u_n}\|_{\mathcal{D}^{1,2}}^2 &\leq &
\int_{\mathbb{R}^3} |\nabla\phi_{u_n}|^2 \,dx + \int_{\mathbb{R}^3}
\phi_{u_n}^2 u_{n}^2 \,dx \\
 &=& -\omega \int_{\mathbb{R}^3} \phi_{u_n} u_{n}^2 \,dx \leq C
 \omega \|\phi_{u_n}\|_{\mathcal{D}^{1,2}}
 \|u_{n}\|_{\frac{12}{5}}^2.
\end{eqnarray*}

So, passing to a subsequence if necessary, we may assume
\begin{itemize}
    \item []\hspace{3cm} $u_{n}\rightharpoonup u$,\hspace{0.3cm}
    weakly in $E$, \hspace{0.1cm} as $n\rightarrow\infty$,
    \item [] \hspace{3cm} $\phi_{u_n}\rightharpoonup \varphi$,\hspace{0.3cm}
    weakly in
    $\mathcal{D}^{1,2}$, \hspace{0.1cm} as $n\rightarrow\infty$.
\end{itemize}

\begin{lemma}\label{existence-of-Phi} $\varphi=\phi_u$.
\end{lemma}

\begin{proof}
The proof is an easy adaptation of \cite{Azzollini-Pomponio-KGM}, but for the sake of completeness we give a sketch of it.

We have that
\begin{eqnarray}
&& u_{n}\rightharpoonup u \quad \text{weakly in}\, L^{s}(\mathbb{R}^3),\hspace{.17cm} 2\leq s\leq 6\nonumber\\
 && u_{n}\rightarrow u \quad \text{in}\, L^{s}_{loc}(\mathbb{R}^3),\hspace{.17cm} 2\leq s < 6\label{loc}.
\end{eqnarray}

Since $\phi_{u_n}\rightharpoonup \varphi$,\hspace{0.3cm} weakly in $\mathcal{D}^{1,2}$, \hspace{0.1cm} as $n\rightarrow\infty$, then
\begin{eqnarray}
&&\phi_{u_n}\rightharpoonup \varphi \quad \text{weakly in}\,\,
L^{6}(\mathbb{R}^3)\nonumber\\
&&\phi_{u_n}\rightarrow\varphi\quad \text{in}\,\,
L^{s}_{loc}(\mathbb{R}^3), \hspace{.17cm} 1\leq s<6 \label{loc2}.
\end{eqnarray}

It remains to show that $\varphi=\phi_{u}$. By Proposition (\ref{Propriedade-phi}), it suffices to show that $\varphi$ satisfies
$\Delta\varphi=(\omega+\varphi)u^2$.

Let $\psi\in C_0^{\infty}(\mathbb{R}^3)$ be a test function. Since $\Delta\phi_{u_n}=(\omega+\phi_{u_n})u_n^2$, we have
\begin{eqnarray*}-\int_{\mathbb{R}^3}\langle\nabla\phi_{u_n},
\nabla\psi\rangle\,dx=\int_{\mathbb{R}^3}\omega\psi u_n^2\,
dx+\int_{\mathbb{R}^3}\phi_{u_n}\psi u_n^2.
\end{eqnarray*}

From (\ref{loc}), (\ref{loc2}) and the boundedness of ($\phi_{u_n})$ in $\mathcal{D}^{1,2}$, the following three sentences hold
\begin{eqnarray}\label{A}\begin{array}{rcl}
\displaystyle\int_{\mathbb{R}^3}\langle\nabla\phi_{u_n},\nabla\psi\rangle\,dx
&\stackrel{ n\rightarrow\infty}{\longrightarrow} &
\displaystyle\int_{\mathbb{R}^3}\langle\nabla\varphi,\nabla\psi\rangle\,dx\\
\displaystyle\int_{\mathbb{R}^3}\phi_{u_n} u_n^2\psi\,dx
&\stackrel{ n\rightarrow\infty}{\longrightarrow}
\displaystyle & \displaystyle \int_{\mathbb{R}^3}\varphi u^2\psi\,dx\\
\displaystyle\int_{\mathbb{R}^3} u_n^2\psi\,dx &\stackrel{
n\rightarrow\infty}{\longrightarrow}
&\displaystyle\int_{\mathbb{R}^3}u^2\psi\,dx
\end{array}
\end{eqnarray}
proving that $\varphi=\phi_{u}$.

\end{proof}

Consider $v\in C_0^\infty(\mathbb{R}^3)$. Using Hölder's inequality and the limitation of ($\phi_{u_n}$), we get
\begin{eqnarray}\label{eq6}
\int_{\mathbb{R}^3}(\phi_{u_n}u_n-\phi_{u}u)v\, dx &=&
\int_{\mathbb{R}^3}\phi_{u_n}(u_n-u)v\, dx +
\int_{\mathbb{R}^3}u(\phi_{u_n}-\phi_u)v\, dx\nonumber\\
&\leq & C\|\phi_{u_n}\|_{\mathcal{D}^{1,2}}\Big(
\int_{\Omega}|u_n-u|^{\frac{6}{5}}
|v|^{\frac{6}{5}}\,dx
\Big)^{\frac{5}{6}}+\nonumber\\
&&+\int_{\Omega}(\phi_{u_n}-\phi_u)uv\,dx\nonumber\\
&=& o_n(1)
\end{eqnarray}
and
\begin{eqnarray}\label{eq7}
\int_{\mathbb{R}^3}(\phi_{u_n}^2 u_n-\phi_{u}^2 u)v\, dx &=&
\int_{\mathbb{R}^3}\phi_{u_n}^2(u_n-u)v\, dx +
\int_{\mathbb{R}^3}u(\phi_{u_n}^2-\phi_u^2)v\, dx\nonumber\\
&\leq & C \|\phi_{u_n}\|_{\mathcal{D}^{1,2}}\Big(
\int_{\Omega}|u_n-u|^{\frac{3}{2}}|v|^{\frac{3}{2}}\,dx
\Big)^{\frac{2}{3}}+\nonumber\\
&&+\int_{\Omega}(\phi_{u_n}^2-\phi_u^2)u v\,dx\nonumber\\
&=& o_n(1),
\end{eqnarray}
where $\Omega$ is the support of the function $v$.

Therefore,
\begin{eqnarray*}\int_{\mathbb{R}^3}(2\omega+\phi_{u_n})\phi_{u_n}u_nv\,dx-\int_{\mathbb{R}^3}(2\omega+\phi_{u})\phi_{u}uv\,dx=o_n(1),
\end{eqnarray*}
for all $ v\in C_0^\infty(\mathbb{R}^3)$, which implies
$$I'(u)v=0, \quad \mbox{for all} \,\,v\in E.
$$

Consequently, $u$ is a weak solution for (\ref{equacao1}).

In view of the lack of compactness, we must prove that actually $u$ does not vanish. To this end we suppose, conversely, that $u\equiv 0$.

We claim that only one of the following conditions hold
\begin{itemize}
\item[(i)] For all $q\in(2,6)$
$$\lim_{n\rightarrow+\infty}\int_{\mathbb{R}^3}|u_n|^q\,dx=0.
$$
\item[(ii)] There are positive numbers $\rho$, $\eta$ and a sequence $(y_n)\subset\mathbb{R}^3$ such that
$$\liminf_{n\rightarrow +\infty}\int_{B_\rho(y_n)}|u_n|^2\,dx>\rho>0.
$$
\end{itemize}

First suppose (ii) holds. If (i) also occurs, from conditions $(f_2)$ and $(f_3)$, we would have
$$\lim_{n\rightarrow+\infty}\int_{\mathbb{R}^3}f(u_n)u_n\,dx=0
$$
and
$$\|u_n\|^2=\int_{\mathbb{R}^3}f(u_n)u_n\,dx+o_n(1),
$$
since
\begin{eqnarray*}-\int_{\mathbb{R}^3}(2\omega+\phi_{u_n})\phi_{u_n}u_n^2\,dx&=&-\int_{\mathbb{R}^3}\omega\phi_{u_n}u_n^2\,dx+
\int_{\mathbb{R}^3}|\nabla\phi_{u_n}|^2\,dx\\
&\leq & 2\omega\|\phi_{u_n}\|_6\|u\|_{\frac{12}{5}}^2\\
&=& o_n(1).
\end{eqnarray*}

As a consequence, the sequence $(u_n)$ would be strongly convergent to 0 in $E$, implying that $I(u_n)\rightarrow 0$, contrary to $I(u_n)\rightarrow c>0$,
as $n\rightarrow +\infty$. Hence, (i) can not occur.

On the other hand, if (ii) does not hold, then there exists $\bar{r}>0$ such that
$$\lim_{n\rightarrow +\infty}\sup_{y_n\in\mathbb{R}^3}\int_{B_{\bar{r}}(y_n)}|u_n|^2\,dx=0
$$
and then $\int_{\mathbb{R}^3}|u_n|^{\alpha}\,dx\rightarrow 0$, for all $\alpha\in(2,6)$, which implies that (i) occurs.

Now, define $\tilde{u}_n(x)=u_n(x+y_n)$. From condition $(V_2)$ we can assume that $y_n\in\mathbb{Z}^3$. Moreover, $(\tilde{u}_n)$ is bounded in $E$ and we
can clearly assume that $(\tilde{u}_n)$ is weakly convergent to $\tilde{u}$, for some $\tilde{u}\in E$. From (ii), $\tilde{u}\neq 0$.

Note that because
$$\int_{\mathbb{R}^3}(2\omega+\phi_{\tilde{u}})\phi_{\tilde{u}}\tilde{u}^2\,dx=
\int_{\mathbb{R}^3}(2\omega+\phi_{u})\phi_{u}u^2\,dx
$$
we have
\begin{eqnarray*}I'(\tilde{u}_n)\tilde{u}_n=I'(u_n)u_n\quad \mbox{and}\quad I(\tilde{u}_n)=I(u_n),
\end{eqnarray*}
hence $(\tilde{u}_n)$ is a Cerami sequence for $I$, and finally
\begin{eqnarray*}I'(\tilde{u})=0\,\,\,\, \mbox{with}\,\,\tilde{u}\neq 0.
\end{eqnarray*}

It follows that $(\tilde{u},\phi_{\tilde{u}})$ is a nontrivial solution for the $(\mathcal{KGM})$ system and then $I(\tilde{u})\geq c$. By using
bootstrap arguments and the maximum principle, we can conclude that $\tilde{u}$ is positive.

Note that from ($f_5'$),
\begin{eqnarray*} 4I(\tilde{u}_n)-I'(\tilde{u}_n)\tilde{u}_n\geq \|u_n\|^2, \quad \mbox{for all}\,\, n\in\mathbb{N}.
\end{eqnarray*}

Passing to the limit we obtain
\begin{eqnarray*} 4c=\liminf_{n\rightarrow +\infty}(4I(\tilde{u}_n)-I'(\tilde{u}_n)\tilde{u}_n)\geq \|u\|^2
\end{eqnarray*}
then, $\|u\|^2\leq 4c$.

Observing that until now we have used condition $(f'_5)$, thus the proof of Theorem \ref{teoremaB} is complete.

Finally, in order to verify that $\tilde{u}$ is a ground state solution, we observe that from $(f_5)$ and Fatou's Lemma,
\begin{eqnarray*}4c &=&\liminf_{n\rightarrow +\infty}(4I(\tilde{u}_n)-I'(\tilde{u}_n)\tilde{u}_n)\\
&=& \liminf_{n\rightarrow +\infty}\Big[ \|\tilde{u}_n\|^2+\int_{\mathbb{R}^3}\phi_{\tilde{u}_n}^2\tilde{u}_n^2\,dx+
\int_{\mathbb{R}^3}H(\tilde{u}_n)\,dx \Big]\\
&\geq &
\|\tilde{u}\|^2 +\int_{\mathbb{R}^3}\phi_{\tilde{u}}^2\tilde{u}^2\,dx+
\int_{\mathbb{R}^3}H(\tilde{u})\,dx\\
&=& 4I(\tilde{u})-I'(\tilde{u})\tilde{u}=4I(\tilde{u})\geq 4c.
\end{eqnarray*}

Hence, $I(\tilde{u})=c$ and so $\tilde{u}$ is a ground state solution of equation (\ref{equacao1}).

\begin{proof}
Now we will prove that there exists a solution $w\in E$ such that $I(w)=m$ where
$$m=\inf\{I(u);\,\,u\neq 0 \,\,\text{and}\,\, I'(u)=0\}.
$$

We first note that $m$ belongs to the interval $(0,c]$ where $c$ is the mountain pass level. Indeed, from $(f'_5)$,
\begin{eqnarray*}4I(u)=4I(u)-I'(u)u=\|u\|^2+\int_{\mathbb{R}^3}\phi_u^2 u^2\,dx+\int_{\mathbb{R}^3}[uf(u)-4F(u)]\,dx>0,
\end{eqnarray*}
for any critical point of $I$. Thus, $m>0$.

On the other hand, in the proof of Theorem \ref{teoremaC}, we obtained a nontrivial critical point $u$ for $I$ as a weak limit of a bounded
$(PS)_c$ sequence $(u_n)$. Then from $(f'_5)$ and Fatou's Lemma,
\begin{eqnarray*}4I(u)&=& 4I(u)-I'(u)u=\|u\|^2+\int_{\mathbb{R}^3}\phi_u^2 u^2\,dx+\int_{\mathbb{R}^3}[uf(u)-4F(u)]\,dx\\
&\leq & \liminf_{n\rightarrow\infty}[4I(u_n)-I'(u_n)u_n]\\
&=& 4c,
\end{eqnarray*}
from what we conclude that $0<m\leq c$.

Now let $(v_n)$ be a sequence of nontrivial critical points of $I$ such that
$$I(v_n)\rightarrow m\in(0,c].
$$

Since $I$ has the mountain pass geometry, we have $\liminf_{n\rightarrow\infty}\|v_n\|\geq r>0$. Then arguing as in the previous sections, $(v_n)$ is
a bounded sequence which converges to $w\in E$, $w\neq 0$. Since $w$ must be a nontrivial critical point of $I$, then $I(w)\geq m$. But, again from
$(f'_5)$ and Fatou's Lemma, $$I(w)\leq\liminf_{n\rightarrow\infty}I(v_n)=m.
$$

Therefore, $I(w)=m$ and this completes the proof of Theorem \ref{teoremaB2}.

\end{proof}

\section{The supercritical case}\label{The supercritical case}

In this section we proof Theorem \ref{teoremaC}.

The following results establish an estimate involving the $L^{\infty}$ norm of a solution to a subcritical problem.

\begin{proposition}\label{estimate} Let $v\in H^1(\mathbb{R}^N)$ be a weak solution of the problem
\begin{eqnarray*}-\Delta v+b(x)v=h(x,v), \quad \text{in}\,\,\mathbb{R}^N,\, N\geq 3,
\end{eqnarray*}
where $h:\mathbb{R}^N\times\mathbb{R}^N\rightarrow \mathbb{R}^N$ is a continuous function verifying, for some $2<q<2^*=2N/(N-2)$, $|h(x,s)|\leq 2|s|^{q-1}$,
$\forall s>0$, and $b$ is a non-negative function in $\mathbb{R}^N$. Then, for all $C>0$, there exists a constant $k=k(q,C)>0$ such that if $\|v\|^2\leq C$,
then $\|v\|_{\infty}\leq k$.
\end{proposition}

\begin{proof} The proof can be found in \cite[Proposition 2.1]{Alves-Souto-Soares-supercritical}.
\end{proof}

In order to establish the existence solution asserted by Theorem \ref{teoremaC}, we assume conditions $(F_1)-(F_6)$ and define a sequence of functions
 $(g_n)$ setting
\begin{eqnarray*}g_n(s)= \left \{ \begin{array}{ll}
 0, & \mbox{if}\,\,\, s\leq 0\\
 g(s), & \mbox{if}\,\,\, 0\leq s\leq M_n\\
 \frac{g(M_n)}{M_n^{q-1}}s^{q-1}, & \mbox{if}\,\,\, M_n\leq s.
\end{array}\right.
\end{eqnarray*}

Consider that the problem
\begin{eqnarray}\label{kgm2} \left \{ \begin{array}{ll}
-\Delta u+ V(x)u-(2\omega +\phi)\phi u=f_{\lambda,n}(u), &  \mbox{in}\quad \mathbb{R}^{3}\\
\hspace{0.26cm} \Delta \phi=(\omega+\phi)u^{2}, & \mbox{in}\quad \mathbb{R}^{3}
\end{array}\right.
\end{eqnarray}
which is variational for every $\lambda>0$ and $n\in\mathbb{N}$, because from $(F_3)$ and $(F_6)$, $f_{\lambda,n}(s)=f_0(s)+\lambda g_n(s)$ satisfies
\begin{eqnarray}\label{f-lambdan} |f_{\lambda,n}|\leq (1+\lambda g(M_n) M_n)|s|^{q-1}.
\end{eqnarray}

Using a reduction method, the functional $I_{\lambda,n}:E\rightarrow\mathbb{R}$ defined by
\begin{eqnarray*} I_{\lambda,n}(u)=\frac{1}{2}\int_{\mathbb{R}^3}(|\nabla u|^2+V(x)u^2)\,dx-
\frac{1}{2}\int_{\mathbb{R}^3}\omega\phi_u u^2\,dx-\int_{\mathbb{R}^3}F_{\lambda,n}(u)\,dx
\end{eqnarray*}
is the Euler-Lagrange functional associated with (\ref{kgm2}). From (\ref{f-lambdan}), $I_{\lambda,n}\in
C^1(E,\mathbb{R})$ with Gateaux derivative given by
\begin{eqnarray*}I_{\lambda,n}'(u)v=\int_{\mathbb{R}^3}(\nabla u\cdot \nabla v+V(x)uv)\,dx-\int_{\mathbb{R}^3}(2\omega +\phi)\phi uv\,dx
-\int_{\mathbb{R}^3}f_{\lambda,n}(u)v\,dx,
\end{eqnarray*}
for every $u, v\in E$.

Now we introduce an auxiliary functional $I_0:E\rightarrow\mathbb{R}$ given by
\begin{eqnarray*} I_{0}(u)=\frac{1}{2}\int_{\mathbb{R}^3}(|\nabla u|^2+V(x)u^2)\,dx-
\frac{1}{2}\int_{\mathbb{R}^3}\omega\phi_u u^2\,dx-\int_{\mathbb{R}^3}F_{0}(u)\,dx,
\end{eqnarray*}
where $F_0$ is defined in $(F_4)$.

Once $f_0$ satisfies $(F_1)-(F_4)$ and from Proposition \ref{Propriedade-phi}, $I_0$ possesses the geometric hypothesis of the mountain
pass theorem (see Lemma \ref{MPG}). Then, there exists $e\in H^{1}(\mathbb{R}^3)$ and $\displaystyle c_0:=\inf_{\gamma\in\Gamma} \max_{t\in [0,1]}I_0(\gamma(t))$
where
\begin{eqnarray}\label{gamma}\Gamma=\{\gamma:[0,1]\rightarrow E : \gamma(0)=0, \gamma(1)=e \}.
\end{eqnarray}

Since the function $f_{\lambda,n}$ satisfies conditions $(f_1)-(f_4)$ and $(f_5')$, for every $\lambda>0$, $n\in\mathbb{N}$, and $V$
satisfies $(V_1)-(V_2)$,
then, from Theorem \ref{teoremaB}, the system (\ref{kgm2}) has a positive solution $u_{\lambda,n}\in E$ such that
\begin{eqnarray*}\|u_{\lambda,n}\|^2\leq 4c_{\lambda,n},
\end{eqnarray*}
where $\displaystyle c_{\lambda,n} =\inf_{\gamma\in\Gamma} \max_{t\in [0,1]}I_{\lambda,n}(\gamma(t))$ and $\Gamma$ is defined by (\ref{gamma}),
which is independent of $\lambda$ and $n$. Indeed, from $(F_1)$, we have $F_{\lambda,n}(s)\geq F_0(s)$, for all $s$. Hence,
\begin{eqnarray}\label{eq8} I_{\lambda,n}(v)\leq I_0(v),
\end{eqnarray}
for all $v\in E$. In particular, $I_{\lambda,n}(e)\leq I_0(e)<0$. Thus $\Gamma$ is independent of $\lambda$ and $n$. Besides, from (\ref{eq8}) we have
\begin{eqnarray}\label{eq9} c_{\lambda,n}\leq c_0.
\end{eqnarray}

Now we are ready to proof Theorem \ref{teoremaC}.

Consider $k=k(q,4 c_0)$ given by Proposition \ref{estimate} and fix $n$ such that $k < M_n$. Let $\lambda_0$ be such that $\lambda_0 g(Mn)M_n\leq 1$.

Note that from (\ref{f-lambdan}), we have
\begin{eqnarray}\label{eq10} |f_{\lambda,n}(s)|\leq 2|s|^{q-1}
\end{eqnarray}
for all $s$ and $\lambda<\lambda_0$.

Using Theorem \ref{teoremaB}, there exists a positive solution $u=u_{\lambda,n}$ of (\ref{kgm2}) such that $\|u\|^2\leq 4c_{\lambda,n}$.
Then, from inequality (\ref{eq9}), we conclude
\begin{eqnarray}\label{eq11}\|u\|^2\leq 4c_0.
\end{eqnarray}

In view of Proposition \ref{Propriedade-phi} and condition $(V_1)$, we have that
\begin{eqnarray}\label{eq12} b(x):=V(x)-(2\omega+\phi_u)\phi_u\geq 0.
\end{eqnarray}

Finally, from (\ref{eq10}), (\ref{eq11}) and (\ref{eq12}), we can apply Proposition \ref{estimate} to obtain $\|u\|_{\infty}\leq K$, for some
$K(q,c_0)$, and the proof of Theorem \ref{teoremaC} is complete.

\paragraph{\textbf{Acknowledgements}}
The author is grateful to Professor S. H. M. Soares for enlightening discussions and helpful suggestions.

\bibliography{<your-bib-database>}

\end{document}